\documentclass[11pt]{article}

\usepackage[a4paper,margin=1in]{geometry}
\usepackage{amsmath,amssymb,amsthm,mathtools}
\usepackage{microtype}
\usepackage{xcolor}
\usepackage{tikz}
\usepackage{float}
\usetikzlibrary{arrows.meta,decorations.pathreplacing}
\usepackage{hyperref}
\usepackage[nameinlink,capitalise]{cleveref}

\hypersetup{colorlinks=true,linkcolor=blue!50!black,citecolor=blue!50!black,urlcolor=blue!50!black}

\newtheorem{theorem}{Theorem}[section]
\newtheorem{proposition}[theorem]{Proposition}

\crefname{theorem}{Theorem}{Theorems}
\crefname{proposition}{Proposition}{Propositions}
\crefname{lemma}{Lemma}{Lemmas}
\crefname{remark}{Remark}{Remarks}

\newcommand{\E}{\mathsf E}
\newcommand{\Prob}{\mathsf P}
\newcommand{\1}{\mathbf 1}
\newcommand{\law}{\mathcal L}
\newcommand{\dd}{\mathrm d}

\title{On a moment determinacy conjecture of Bertoin and Yor}

\author{
Martin Minchev\thanks{Institute of Mathematics, University of Zurich, Switzerland. Faculty of Mathematics and Informatics, Sofia University, Bulgaria. Email: \texttt{martin.minchev@math.uzh.ch}.}
}

\date{}

\begin{document}
\maketitle

\begin{abstract}
Let $\xi$ be an unkilled real-valued L\'evy process which drifts to $+\infty$ and has positive exponential moments of all orders, and define
\[
 I_\xi=\int_0^\infty e^{-\xi_t}\,\dd t,
 \qquad
 \text{and its reciprocal }
 X_\xi=1/I_\xi.
\]
Bertoin and Yor \cite{Bertoin-Yor-2002} proved that $X_\xi$ is moment-determinate when $\xi$ has no positive jumps, and conjectured that this condition is also necessary. We prove the latter.

The proof is based on a lower bound near zero for the law of $I_\xi$. We show that a group of sufficiently many positive jumps near the origin puts $I_\xi$ on a suitable small scale. The first selected jump time is used as a one-dimensional smooth coordinate, yielding an absolutely continuous subcomponent of the law of $I_\xi$. After the change of variables, the resulting subdensity of $X_\xi$ satisfies a Krein moment indeterminacy criterion.
\end{abstract}


\section{Introduction}

Let $\xi$ be an unkilled L\'evy process with Laplace exponent
\[
 \psi(q)=\log \E(e^{q\xi_1}),
 \qquad
 \text{for }q\geq0,
\]
and suppose throughout that
\begin{equation}\label{eq:standing}
 \psi(q)<\infty\quad\text{for all }q\geq0,
 \qquad
 0<\psi'(0+)=\E(\xi_1)<\infty.
\end{equation}
Then $\xi$ drifts to $+\infty$, and the exponential functional
\[
 I_\xi=\int_0^\infty e^{-\xi_t}\,\dd t\qquad
 \text{is finite a.s. }
\]
Bertoin and Yor \cite{Bertoin-Yor-2002} proved that its reciprocal has all moments given by
\begin{equation}\label{eq:BY}
 \E(I_\xi^{-k})=
 \psi'(0+)\frac{\psi(1)\psi(2)\cdots\psi(k-1)}{(k-1)!},
 \qquad
 \text{for }k\geq1,
\end{equation}
and showed that, when $\xi$ has no positive jumps, $1/I_\xi$ has an exponential moment in a neighbourhood of the origin and is therefore moment-determinate. They conjectured that the absence of positive jumps is also necessary; see also \cite{Bertoin-Yor-2005,Minchev-Savov-2026}.

The aim of this note is to prove this necessity. If positive jumps are present, then $\psi(q)$ grows at least exponentially along large $q$, so classical Carleman-type criteria fail; for background see, for example, \cite{Lin-2017,Stoyanov-Lin-Kopanov-2020}. However, this does not imply moment indeterminacy on its own, and one needs additional information on the density of $I_\xi$ near zero.
    One possible approach, as described in \cite[Remark 3.7]{Minchev-Savov-2026}, would be to use Mellin inversion and saddle-point analysis of the Bernstein--gamma functions introduced by Patie and Savov~\cite{Patie-Savov-2018}, in the spirit of the density asymptotics obtained in~\cite{Minchev-Savov-2023}. Here we obtain the required lower bound probabilistically.
   In the pure Poisson case, Bertoin and Yor \cite[Section~4]{Bertoin-Yor-2002} obtained explicit small-value density information for $I_\xi$. After inversion, this gives a lognormal-type tail for $1/I_\xi$, in the classical scale of moment indeterminacy. Here we prove a lower bound of the same logarithmic order.

Consequently, a subdensity version of the Krein criterion, Proposition \ref{prop:subkrein}, provides the main result. 

\begin{theorem}\label{thm:main}
Assume \eqref{eq:standing}. Then
\[
 1/I_\xi\text{ is moment-determinate}
 \quad\iff\quad
 \Pi((0,\infty))=0,
\]
where $\Pi$ denotes the Lévy measure of $\xi$.
\end{theorem}

Only the forward direction has not been proved. The argument can be summarised as follows: if $m$ positive jumps of size at least $a$ occur near time zero, then the future part of the exponential functional is multiplied by roughly $e^{-am}$. Thus a natural scale for small values of $I_\xi$ is \[
 \ell_m=L e^{-am},
 \qquad \text{for some $L>0$}.
\]
The number $m$ of jumps makes the future contribution small. The first selected jump time $t$ is then varied in an interval of order $\ell_m$. Before this first selected jump, the selected compound Poisson part has not moved, so the contribution to $I_\xi$ is of order $t$ and \textit{a fortiori} $\ell_m$ as well. By choosing  suitably endpoints of the interval in which $t$ changes, the monotone map $t\mapsto I(t)$ crosses the whole
\[
\mathcal I_m=(\ell_m,\ell_{m-1}].
\]
Since the first selected jump time has a density and this map has a uniform one-sided Lipschitz bound on the controlled event, this gives an absolutely continuous subcomponent of the law of $I_\xi$ on each $\mathcal I_m$. Thus we obtain a density lower bound.

The mechanism is sketched in the figure below. The constants
$r_0$, $r_1$, $\delta$ and $R$ are fixed later in the proof; here the figure
only indicates their roles.
\begin{figure}[H]
\centering

\begin{minipage}[t]{0.47\textwidth}
\vspace{0pt}
\raggedright\textbf{(a) selected jumps}

\vspace{1mm}
\centering
\begin{tikzpicture}[x=0.82cm,y=0.85cm,>=Latex,font=\scriptsize]
\path[use as bounding box] (-0.15,0.20) rectangle (7.25,2.05);

% time axis
\draw[->] (0,1.10) -- (7.05,1.10) node[right] {time};

% ticks
\draw (0,1.00) -- (0,1.20) node[below=4pt] {$0$};
\draw (1.25,1.00) -- (1.25,1.20) node[below=4pt] {$r_0\ell_m$};
\draw (2.65,1.00) -- (2.65,1.20) node[below=4pt] {$t$};
\draw (4.65,1.00) -- (4.65,1.20) node[below=4pt] {$r_1\ell_m$};
\draw (6.10,1.00) -- (6.10,1.20) node[below=4pt] {$R\ell_m$};

% selected jump crosses on the time axis
\draw[red!75!black,thick]
  (2.65,1.10) +(-0.11,-0.11) -- +(0.11,0.11)
  (2.65,1.10) +(-0.11,0.11) -- +(0.11,-0.11);

\foreach \x in {2.9,3.1,3.35,3.75,4.0}{
  \draw[red!75!black,thick]
    (\x,1.10) +(-0.08,-0.08) -- +(0.08,0.08)
    (\x,1.10) +(-0.08,0.08) -- +(0.08,-0.08);
}

% short interval containing the group after the first jump
\draw[blue!70!black,very thick] (2.65,0.52) -- (4.3,0.52);
\draw[blue!70!black] (2.65,0.44) -- (2.65,0.60);
\draw[blue!70!black] (4.3,0.44) -- (4.3,0.60);
\node[blue!70!black] at (3.475,0.26) {$\delta\ell_m$};

\draw[blue!70!black,very thick] (4.65,1.1) -- (6.1,1.1);
\draw[blue!70!black] (4.65,0.44) -- (4.65,0.60);
\draw[blue!70!black] (6.1,0.44) -- (6.1,0.60);
% \node[blue!70!black] at (3.10,0.26) {$\delta\ell_m$};

% % short interval2 containing the group after the first jump
% \draw[blue!70!black,very thick] (5.25,0.) -- (3.55,0.52);
% \draw[blue!70!black] (2.65,0.44) -- (2.65,0.60);
% \draw[blue!70!black] (5.55,0.44) -- (3.55,0.60);
% \node[blue!70!black] at (3.10,0.26) {$\delta\ell_m$};

% deterministic cut
\draw[red!70!black,dashed,thick] (6.10,0.48) -- (6.10,1.82);
\node[red!70!black,above] at (6.10,1.82) {cut};
\end{tikzpicture}
\end{minipage}
\hfill
\begin{minipage}[t]{0.47\textwidth}
\vspace{0pt}
\raggedright\textbf{(b) density mechanism}

\vspace{1mm}
\centering
\begin{tikzpicture}[x=0.82cm,y=0.85cm,>=Latex,font=\scriptsize]
\path[use as bounding box] (-0.75,0.20) rectangle (7.05,2.05);

% axes
\draw[->] (0,0.45) -- (6.85,0.45) node[right] {$t$};
\draw[->] (0,0.45) -- (0,1.95) node[left] {$F_m$};

% target levels
\draw[dashed] (0,0.92) -- (6.15,0.92);
\draw[dashed] (0,1.62) -- (6.15,1.62);
\node[left] at (0,0.92) {$\ell_m$};
\node[left] at (0,1.62) {$\ell_{m-1}$};

% x-axis ticks
\draw (1.30,0.35) -- (1.30,0.55) node[below=4pt] {$r_0\ell_m$};
\draw (5.25,0.35) -- (5.25,0.55) node[below=4pt] {$r_1\ell_m$};

% increasing curve
\draw[very thick,blue!70!black]
  plot[smooth] coordinates {
    (1.30,0.72)
    (2.20,0.85)
    (3.20,1.06)
    (4.15,1.32)
    (5.25,1.64)
    (6.00,1.82)
  };
\end{tikzpicture}
\end{minipage}

\caption{The group of jumps construction. Left: the first selected jump time $T$ is allowed to take values $t\in[r_0\ell_m,r_1\ell_m]$, and the remaining selected jumps occur within a time interval of length $\delta\ell_m$ after it, before the deterministic cut time $R\ell_m$. Right: after the remaining randomness is frozen, the map $t\mapsto F_m(t)$ is increasing and crosses the scale interval $\mathcal I_m=(\ell_m,\ell_{m-1}]$.}
\label{fig:group}
\end{figure}

% For completeness, we note note that the law of $I_\xi$, that is $\mathcal L(I_\xi)$,  has a density in non-deterministic cases by \cite{Pardo-Rivero-VanSchaik-2013}, see also
% \cite[Section]{Minchev-Savov-2026}. However we need only a subdensity.

\subsection{Use of AI tools}

OpenAI's ChatGPT was used during the exploratory and editorial stages of this work, including GPT-5.4 Thinking and GPT-5.5 variants such as Thinking and Pro. In particular, the idea of trying a direct construction based on an early group of positive jumps, and then varying the time of the first selected jump as a smooth coordinate, came up in a \textit{discussion} with it. 

The use of a selected jump time can be seen in Bertoin, Lindner, and Maller~\cite[Section 3.3]{Bertoin-Lindner-Maller-2008}: they condition on almost all random quantities except certain jump times, so that the remaining conditional law is obtained as the image of Lebesgue measure under a deterministic map. The author takes full responsibility for the manuscript.

\section{A Krein criterion for subdensities}

    We first record the corollary of the usual Krein criterion which we use later. If $\mu$ is a probability measure, we call a finite positive measure $\nu$  a \textit{submeasure} of $\mu$ if $\nu(A)\leq\mu(A)$ for every measurable set $A$.

\begin{proposition}[Krein criterion for subdensities]\label{prop:subkrein}
Let $\mu$ be a probability measure on $(0,\infty)$ with moments of all positive orders. Suppose that there exists a non-zero absolutely continuous submeasure
\[
 \nu(\dd x)=h(x)\,\dd x\leq \mu(\dd x)
\]
such that $\nu$ has moments of all positive orders and
\begin{equation}\label{eq:krein}
 \int_c^\infty \frac{-\log h(x^2)}{1+x^2}\,\dd x<\infty,
 \qquad
 \text{ for some $c>0$,}
\end{equation}
where $h$ is positive on $(c^2,\infty)$. Then $\mu$ is moment-indeterminate.
\end{proposition}

\begin{proof}
Define the probability measure
$
 \widetilde\nu(\dd x)=\nu(\dd x)/\nu((0,\infty)).
$
Multiplying the density by a positive constant does not affect the finiteness of the logarithmic integral. Hence the half-line form of Krein's criterion \cite[Theorem~4]{Lin-2017} applies to $\widetilde\nu$. Thus there exists a probability measure $\widetilde\nu_\bullet\ne\widetilde\nu$ on $(0,\infty)$ such that
\[
 \int_0^\infty x^n\,\widetilde\nu(\dd x)
 =
 \int_0^\infty x^n\,\widetilde\nu_\bullet(\dd x),
 \qquad
 \text{for }n\geq0.
\]
Writing
$
 \mu=\nu+\mu_0,
$
with $\mu_0$ a positive finite measure, define
$ \widetilde\mu=\nu((0,\infty))\,\widetilde\nu_\bullet+\mu_0.
$
Then $\widetilde\mu$ is a probability measure whose moments coincide with those of $\mu$, while $\widetilde\mu\ne\mu$. Therefore $\mu$ is moment-indeterminate.
\end{proof}

\section{Density bound in the positive-jump case}
Assume in this subsection that
\[
 \Pi((0,\infty))>0,
 \quad \text{and choose $a>0$ such that }
 \lambda=\Pi([a,\infty))>0.
\]
Writing
$
 \Delta\xi_s=\xi_s-\xi_{s-}
$
for the jump of $\xi$ at time $s$, set
\[
 S_t=\sum_{s\leq t}\Delta\xi_s\,\1_{\{\Delta\xi_s\geq a\}},
 \qquad\text{and}\qquad
 \eta_t=\xi_t-S_t.
\]
Then $S$ is a compound Poisson process of rate $\lambda$, independent of the L\'evy process $\eta$.

\begin{proposition}\label{prop:group}
There exist constants $c_1, c_2>0$ and an absolutely continuous submeasure
\[
 \nu_{I_\xi}(\dd y)=h_{I_\xi}(y)\,\dd y\leq \Prob(I_\xi\in\dd y)
\]
such that
\begin{equation}\label{eq:hI-lower}
 h_{I_\xi}(y)\geq c_1\exp \left(-c_2(\log(1/y))^2\right),
 \qquad
 \text{for all sufficiently small }y>0.
\end{equation}
\end{proposition}

\begin{proof}
For lighter notation, we write $I$ instead of $I_\xi$ in the proof. Fix $K>0$. Since $\eta$ is a c\`adl\`ag L\'evy process starting from $0$, we can choose $A_0>0$ such that
\begin{equation}\label{eq:control-eta}
 p_0=
 \Prob\left(\sup_{0\leq s\leq A_0}|\eta_s|\leq K\right)>0.
\end{equation}
Choose $r_0,r_1,\delta>0$ such that
\begin{equation}\label{eq:constant-choice}
 e^K(r_0+\delta)<1/4,
 \qquad
 e^{-K}r_1>e^a,
 \qquad \text{and} \qquad
 R=r_1+\delta.
\end{equation}
This choice will be used only for the crossing estimate
\[
 F_m(r_0\ell_m)<\ell_m,
 \qquad
 F_m(r_1\ell_m)>\ell_{m-1},
\]
sketched in the right panel of Figure \ref{fig:group}. Next, choose $M>0$ such that $\Prob(I\leq M)>0$. For a positive constant $L$, which we choose shortly, and $m\geq1$, we introduce the natural scale
\begin{equation}\label{eq:scale}
 \ell_m=L e^{-am},
 \qquad
 \mathcal I_m=(\ell_m,\ell_{m-1}].
\end{equation}
At the deterministic cut time $R\ell_m$, the Markov property gives
\begin{equation}\label{eq:markov-split}
 I=
 \int_0^{R\ell_m}e^{-\xi_s}\,\dd s
 +e^{-\xi_{R\ell_m}}I',
\end{equation}
where $I'$ is an independent copy of $I$.
We choose $L$ so that, on $\{I'\leq M\}$, the second term in \eqref{eq:markov-split} is negligible on the scale $\ell_m$ after the $m$ selected jumps. Indeed, on the controlled event, after these jumps have occurred,
\begin{equation}
 e^{-\xi_{R\ell_m}}I'
 \leq e^K e^{-am}M
 =
 \frac{e^K M}{L}\ell_m,
 \quad
 \text{thus choose $L>0$ so that
 $e^K M\leq \frac{L}{8}$}.
\end{equation}
Fix $m$ large enough so that $R\ell_m\leq A_0$, so we can use \eqref{eq:control-eta}. We construct a submeasure of $\law(I)$ with density bounded from below on $\mathcal I_m$. Leave the first selected jump time as a continuous variable
\[
 t\in[r_0\ell_m,r_1\ell_m],
\]
and impose the following selected-jump pattern:
\begin{itemize}
 \item the first jump of $S$ occurs at time $t$;
 \item exactly $m-1$ further jumps of $S$ occur in $(t,t+\delta\ell_m]$;
 \item no other jump of $S$ occurs before $R\ell_m$;
 \item $\displaystyle \sup_{0\leq s\leq R\ell_m}|\eta_s|\leq K$;
 \item in the Markov decomposition at time $R\ell_m$, the independent copy $I'$ satisfies $I'\leq M$.
\end{itemize}
The fourth condition has probability at least $p_0$ by \eqref{eq:control-eta}, and the fifth has probability $\Prob(I\leq M)>0$.

Let $T$ be the first jump time of $S$, and let $N_S(B)$ denote the number of jumps of $S$ in a Borel set $B$. For every Borel set
$
 B\subseteq [r_0\ell_m,r_1\ell_m],
$
the selected-jump conditions give
\begin{equation}\label{eq:q-m}
\begin{aligned}
&\Prob\big(
 T\in B,\,
 N_S((T,T+\delta\ell_m])=m-1,\,
 N_S((T+\delta\ell_m,R\ell_m])=0
 \big)  \\
&\hspace{1.5cm}
 =
 \lambda e^{-\lambda R\ell_m}
 \frac{(\lambda\delta\ell_m)^{m-1}}{(m-1)!}
 \,\mathrm{Leb}(B).
\end{aligned}
\end{equation}
Indeed, conditionally on $T=t$, the increments of $S$ after time $t$ are independent of the past, and the conditional probability of the two count conditions equals
\[
 e^{-\lambda\delta\ell_m}
 \frac{(\lambda\delta\ell_m)^{m-1}}{(m-1)!}
 e^{-\lambda(R\ell_m-t-\delta\ell_m)}.
\]
Multiplying by the density $\lambda e^{-\lambda t}$ of $T$ gives \eqref{eq:q-m}.

Freeze the path of $\eta$ on $[0,R\ell_m]$, the selected jump sizes, the relative positions of the $m-1$ jumps in $(t,t+\delta\ell_m]$, and the value of $I'$, under the two control conditions above. Let $F_m(t)$ be the value of $I$ obtained from this frozen configuration. We now check the crossing property, see the right panel of Figure~\ref{fig:group}, namely
\[
 F_m(r_0\ell_m)<\ell_m,
 \qquad \text{and}\qquad
 F_m(r_1\ell_m)>\ell_{m-1}.
\]
At $t=r_0\ell_m$, the contribution before and during the group is at most $e^K(r_0+\delta)\ell_m$. After the group, all $m$ selected jumps have occurred, so the remaining contribution, including the second term in \eqref{eq:markov-split}, is at most
\[
 e^Ke^{-am}(R\ell_m+M).
\]
By \eqref{eq:constant-choice} and taking $m$ large enough,
\[
 F_m(r_0\ell_m)<\ell_m.
\]
At the other endpoint, before the first selected jump the selected compound Poisson part has not moved. Hence
\[
 F_m(r_1\ell_m)
 \geq
 \int_0^{r_1\ell_m}e^{-\eta_s}\,\dd s
 \geq
 e^{-K}r_1\ell_m
 >
 e^a\ell_m
 =
 \ell_{m-1}.
\]
Thus, once continuity is established, the image of $[r_0\ell_m,r_1\ell_m]$ under $F_m$ contains $\mathcal I_m$.

The map $F_m$ is increasing, since delaying the selected group can only increase $e^{-\xi_s}$ pointwise. Moreover, if the first selected jump time is shifted from $t$ to $t+u$, then each of the $m$ selected jump times is shifted by $u$. Thus the two integrands can differ only on at most $m$ time intervals of length $u$. On the controlled event, $e^{-\xi_s}\leq e^K$, and the terminal term in \eqref{eq:markov-split} is unchanged, since all selected jumps have occurred before $R\ell_m$ in both configurations. Therefore
\begin{equation}\label{eq:lipschitz}
 0\leq F_m(t+u)-F_m(t)\leq e^Kmu,
\end{equation}
whenever $t,t+u\in[r_0\ell_m,r_1\ell_m]$.
In particular $F_m$ is continuous. Since its image contains $\mathcal I_m$, \eqref{eq:lipschitz} implies that, for every Borel set $A\subset\mathcal I_m$,
\begin{equation}\label{eq:preimage}
 \mathrm{Leb}\{t\in[r_0\ell_m,r_1\ell_m]:F_m(t)\in A\}
 \geq
 e^{-K}m^{-1}\mathrm{Leb}(A).
\end{equation}
Now apply \eqref{eq:q-m} with
\[
 B=\{t\in[r_0\ell_m,r_1\ell_m]:F_m(t)\in A\}.
\]
Let $A\subset \mathcal I_m$ be Borel. The estimates above are uniform in the frozen variables satisfying the two control conditions. Hence the contribution to $\Prob(I\in A)$ coming from the selected-jump construction is bounded below by
\[
 p_0\Prob(I\leq M)\,
 \lambda e^{-\lambda R\ell_m}
 \frac{(\lambda\delta\ell_m)^{m-1}}{(m-1)!}
 e^{-K}m^{-1}\mathrm{Leb}(A).
\]
Here $p_0$ and $\Prob(I\leq M)$ come from the two control conditions, \eqref{eq:q-m} gives the selected-jump contribution, and \eqref{eq:preimage} gives the lower bound on the set of first-jump times leading to $I\in A$. The jump sizes and the relative positions of the $m-1$ later jumps are then integrated over their conditional laws.

Since $\ell_m\to0$, the factor $e^{-\lambda R\ell_m}$ is bounded from below for all large $m$. Absorbing all fixed positive constants into $c>0$, we get
\[
 \Prob(I\in A)
 \geq
 c\,m^{-1}\frac{(c\ell_m)^{m-1}}{(m-1)!}\mathrm{Leb}(A),
 \qquad
 \text{for all Borel }A\subset\mathcal I_m.
\]
Equivalently, on $\mathcal I_m$ the law of $I$ dominates an absolutely continuous measure with density $h_m$ satisfying
\begin{equation}\label{eq:hm-lower}
 h_m(y)\geq c\,m^{-1}\frac{(c\ell_m)^{m-1}}{(m-1)!},
 \qquad
 \text{for }y\in\mathcal I_m,
\end{equation}
where $c>0$ is independent of $m$ and $y$.
Consider then
\[
 h_I(y)=\sum_{m\geq m_0}\1_{\mathcal I_m}(y)h_m(y),
\]
where $m_0$ is large enough for the preceding estimates to hold. Since the intervals $\mathcal I_m$ are disjoint and each $h_m(y)\,\dd y$ is dominated by $\law(I)$ on $\mathcal I_m$, the measure $h_I(y)\,\dd y$ is also dominated by $\law(I)$.

It remains to estimate \eqref{eq:hm-lower}. By Stirling's formula and \eqref{eq:scale},
\[
 -\log\left(
 m^{-1}\frac{(c\ell_m)^{m-1}}{(m-1)!}
 \right)
 \leq C_1m^2,
 \qquad \text{for large $m$.}
\]
Moreover, for $y\in\mathcal I_m=(\ell_m,\ell_{m-1}]$ with $\ell_m=Le^{-am}$, we have that
\[
 a(m-1)-\log L
 \leq
 \log(1/y)
 <
 am-\log L.
\]
Thus $m$ is comparable with $\log(1/y)$ uniformly for $y\in\mathcal I_m$. Consequently, there exist constants $c_1,c_2>0$ such that \[ h_I(y)\geq c_1\exp\left(-c_2(\log(1/y))^2\right), \qquad \text{for all sufficiently small }y>0. \]
\end{proof}
\section{Proof of the main result}

Theorem~\ref{thm:main} is now a straightforward consequence of the previous propositions.

\begin{proof}[Proof of Theorem~\ref{thm:main}]
Assume first that $\Pi((0,\infty))>0$, and let $h_{I_\xi}$ be the subdensity from Proposition~\ref{prop:group}. The change of variables $x=1/y$ gives an absolutely continuous submeasure of the law of $X_\xi=1/I_\xi$ with density
\[
 h_{X_\xi}(x)=x^{-2}h_{I_\xi}(1/x).
\]
From \eqref{eq:hI-lower},
\[
 h_{X_\xi}(x)\geq c_1x^{-2}\exp\left(-c_2(\log x)^2\right),
 \qquad
 \text{for large }x.
\]
Consequently, for some $C_1,C_2,C_3\geq0$,
\[
 -\log h_{X_\xi}(x^2)
 \leq
 C_1+C_2\log x+C_3(\log x)^2,
 \qquad
 \text{for large }x.
\]
It follows that, for a suitable $c>0$,
\[
 \int_c^\infty \frac{-\log h_{X_\xi}(x^2)}{1+x^2}\,\dd x<\infty.
\]
Therefore Proposition~\ref{prop:subkrein} implies that $X_\xi$ is moment-indeterminate.

For completeness, we recall the argument for the reverse direction, as in \cite{Bertoin-Yor-2002}. If $\Pi((0,\infty))=0$, then $\psi(q)=O(q^2)$ as $q\to\infty$ by the L\'evy--Khintchine formula. Hence \eqref{eq:BY} gives
\[
 \E(X_\xi^k)=\E(I_\xi^{-k})\leq C^k k!,
 \qquad
 \text{for all }k\geq1.
\]
Thus $\E(e^{\theta X_\xi})<\infty$ for some $\theta>0$, and therefore $X_\xi$ is moment determinate.
\end{proof}
\section*{Acknowledgments}

The author thanks Professor Jordan  Stoyanov for the invitation to the session
``Probability Distributions via Moments and Cumulants'' at the 2026 IMS Annual Meeting in Salzburg. This invitation encouraged the work on this project. The author was supported by the SNSF SCIEX programme, grant IZSF-0-235789.

\bibliographystyle{alpha}
\bibliography{Bibliography}

\end{document}